\PassOptionsToPackage{svgnames}{xcolor}

\documentclass[twoside, fontsize=11pt]{article}
\usepackage{amsmath}
\usepackage{amsthm}
\usepackage{amssymb}
\usepackage{mathtools}
\usepackage{gensymb}
\usepackage{mathrsfs}
\usepackage{physics}
\usepackage[margin=1in]{geometry}
\usepackage{scrextend}
\usepackage{tikz-cd}
\usepackage{titlesec}
\usepackage{titleps}

\titleformat*{\section}{\scshape\filcenter}
\titlelabel{\thetitle.\hspace{.25cm}}

\usepackage[notcolon, notquote]{hanging}

\usepackage{parskip}

\newcommand{\Z}{\mathbb{Z}}

\renewcommand{\b}{\mathfrak{b}}
\renewcommand{\d}{\mathfrak{d}}

\renewcommand{\O}{\mathcal{O}}

\renewcommand{\P}{\mathbb{P}}

\DeclareMathOperator{\mult}{mult}

\DeclareMathOperator{\NS}{NS}
\DeclareMathOperator{\Bl}{Bl}

\DeclareMathOperator{\irr}{irr}

\renewcommand{\hat}[1]{\widehat{#1}}

\makeatletter
\def\thm@space@setup{
	\thm@preskip=0.1in
	\thm@postskip=0in % or whatever, if you don't want them to be equal
}
\makeatother

\theoremstyle{plain}
\newtheorem{thm}{Theorem}[section]
\newtheorem{lem}[thm]{Lemma}
\newtheorem{prop}[thm]{Proposition}
\newtheorem{cor}[thm]{Corollary}

\theoremstyle{definition}

\newtheorem{remark}[thm]{Remark}

\newtheorem{set-up}[thm]{Set-Up}

\title{\Large \textbf{DEGREE OF IRRATIONALITY OF VERY GENERAL ABELIAN SURFACES}}
\author{Nathan Chen}
\date{}

\newpagestyle{chapsec}{
  \sethead[\thepage]% <even left head>
    [\small NATHAN CHEN]% <even centre head>
    []% <even right head>
    {}% <odd left head>
    {\small DEGREE OF IRRATIONALITY OF VERY GENERAL ABELIAN SURFACES}% <odd centre head>
    {\thepage}% <odd right head>
}
\begin{document}

\maketitle

\section{Introduction}

Given a projective variety $X$ of dimension $n$ which is not rational, one can try to quantify how far it is from being rational. When $n = 1$, the natural invariant is the \textit{gonality} of a curve $C$, defined to be the smallest degree of a branched covering $C' \rightarrow \P^{1}$ (where $C'$ is the normalization of $C$). One generalization of gonality to higher dimensions is the \textit{degree of irrationality}, defined as:
\[ \irr(X) = \min \{ \delta > 0 \mid \exists \text{ a degree $\delta$ rational dominant map } X \dashrightarrow \P^{n} \}. \]
Recently, there has been significant progress in understanding the case of hypersurfaces of large degree (cf. \cite{Bastianelli17}, \cite{BCD14}, \cite{BDELU17}). The history behind the development of these ideas is described in \cite{BDELU17}. The results of \cite{Bastianelli17}, \cite{BCD14}, \cite{BDELU17} depend on the positivity of the canonical bundles of the varieties in question, so it is interesting to consider what happens in the $K_{X}$-trivial case. Our purpose here is to prove the somewhat surprising fact that the degree of irrationality of a very general polarized abelian surface is uniformly bounded above, independently of the degree of the polarization.

To be precise, let $A = A_{d}$ be an abelian surface carrying a polarization $L = L_{d}$ of type $(1, d)$ and assume that $\NS(A) = \Z[L]$. An argument of Stapleton \cite{Stapleton17} showed that there is a constant $C$ such that
\[ \irr(A) \leq C \cdot \sqrt{d} \]
for $d \gg 0$, and it was conjectured in \cite{BDELU17} that equality holds asymptotically. Our main result shows that this is maximally false:

\begin{thm}\label{thm:MT}
For an abelian surface $A = A_{d}$ with Picard number $\rho = 1$, one has
\[ \irr(A) \leq 4. \]
\end{thm}

We conjecture that in general equality holds. However, as far as we can see, the conjecture of \cite{BDELU17} for polarized K3 surfaces $(S_{d}, B_{d})$ of genus $d$ - namely, that there exist constants $C_{1}, C_{2}$ such that $C_{1} \cdot \sqrt{d} \leq \irr(S_{d}) \leq C_{2} \cdot \sqrt{d}$ for $d \gg 0$ - remains plausable.\footnote{In other words, $B_{d}$ is an ample line bundle on $S_{d}$ with $B_{d}^{2} = 2d - 2$.}

For an abelian variety $A$ of dimension $n$, it has been shown in \cite{AP92} that $\irr(A) \geq n+1$ (for $n = 2$, one can also see this via Lemma~\ref{lem:NMT}). When $n = 2$, Yoshihara proved that $\irr(A) = 3$ for abelian surfaces $A$ containing a smooth curve of genus 3 (cf. \cite{Yoshihara96}). On a related note, Voisin \cite{Voisin18} showed that the covering gonality of a very general abelian variety $A$ of dimension $n$ is bounded from below by $f(n)$, where $f(n)$ grows like $\log(n)$, and this lower bound was subsequently improved to $\lceil \frac{1}{2} n + 1 \rceil$ by Martin \cite{Martin19}.\footnote{Covering gonality is defined as the minimum integer $c>0$ such that given a general point $x \in A$, there exists a curve $C$ passing through $x$ with gonality $c$.}

In the proof of our theorem, assuming as we may that $L$ is symmetric, we consider the space $H^{0}(A, \O_{A}(2L))^{+}$ of even sections of $\O_{A}(2L)$. By imposing suitable multiplicities at the two-torsion points of $A$, we construct a subspace $V \subset H^{0}(A, \O_{A}(2L))^{+}$ which numerically should define a rational map from $A$ to a surface $S \subset \P^{N}$. Using bounds on the degree of the map and the degree of $S$, as well as projection from linear subspaces, we construct a degree 4 rational covering $A \dashrightarrow \P^{2}$. The main difficulty is to deal with the possibility that $\P_{\text{sub}}(V)$ has a fixed component; this approach was inspired in part by the work of Bauer in \cite{Bauer94}, \cite{Bauer99}.

\textbf{Acknowledgments.} I would like to thank my advisor Robert Lazarsfeld for suggesting the conjecture and for his encouragement and guidance throughout the formulation of the results in this paper. I would also like to thank Frederik Benirschke, Mohamed El Alami, Fran\c{c}ois Greer, Samuel Grushevsky, Ljudmila Kamenova, Yoon-Joo Kim, Radu Laza, John Sheridan, and Ruijie Yang for engaging in valuable discussions.

\section{Set-up}\label{SU}

Let $A = A_{d}$ be an abelian surface with $\rho(A) = 1$. Assume $\NS(A) \cong \Z [L]$ where $L$ is a polarization of type $(1, d)$ for some fixed $d \geq 1$, so that $L^{2} = 2d$ and $h^{0}(L) = d$. Let
\[ \iota: A \rightarrow A, \quad x \mapsto -x \]
be the inverse morphism, and let $Z = \{ p_{1}, \ldots, p_{16} \}$ be the set of two-torsion points of $A$ (fixed points of $\iota$). We may assume that $L$ is symmetric -- that is, $\iota^{\ast}\O_{A}(L) \cong \O_{A}(L)$ -- by replacing $L$ with a suitable translate. In particular, the cyclic group of order two acts on $H^{0}(A, \O_{A}(2L))$. The space of \textit{even} sections $H^{0}(A, \O_{A}(2L))^{+}$ of the line bundle $\O_{A}(2L)$ (sections $s$ with the property that $\iota^{\ast} s = s$) has dimension
\[ h^{0}(A, 2L)^{+} = 2d + 2 \]
(see \cite[Corollary 4.6.6]{BL04}). An even section of $\O_{A}(2L)$ vanishes to even order at any two-torsion point, so we need to impose at most
\[ 1 + 3 + \cdots + (2m-1) = m^{2} \]
conditions for every even section to vanish to order $2m$ at \underline{any} fixed point $p \in Z$ (see \cite{Bauer94} for more details).

Fix any integer solutions $a_{1}, \ldots, a_{16} \geq 0$ to the equation
\[ \sum_{i=1}^{16} a_{i}^{2} = 2d-2, \]
with $a_{15} = 0 = a_{16}$.\footnote{This assumption will be useful in Corollary~\ref{cor:SQ}. For larger values of $d$, note that there are many solutions.} This is possible by Lagrange's four-squares theorem. Let $V \subset H^{0}(A, \O_{A}(2L))^{+}$ be the space of even sections vanishing to order at least $2a_{i}$ at each point $p_{i}$, such that
\[ \dim V \geq 2d + 2 - \sum_{i=1}^{16} a_{i}^{2} \geq 4. \]
Let $\d = \P_{\text{sub}}(V) \subseteq \abs{2L}^{+}$ be the corresponding linear system of divisors, whose dimension is $N \coloneqq \dim \d \geq 3$. Write
\[ d_{i} \coloneqq \mult_{p_{i}} D \]
for a general divisor $D \in \d$, so that $d_{i} \geq 2a_{i}$.

\begin{remark}\label{rem:K}
From \cite[Section 4.8]{BL04}, it follows that sections of $V$ are pulled back from the singular Kummer surface $A/\iota$, so any divisor $D \in \d$ is symmetric, i.e. $\iota(D) = D$.
\end{remark}

Let $\varphi: A \dashrightarrow \P^{N}$ be the rational map given by the linear system $\d$ above, and write $S \coloneqq \overline{\Im(\varphi)}$ for the image of $\varphi$. Regardless of whether or not $\d$ has a fixed component, we find that:

\begin{prop}\label{prop:S}
$S \subset \P^{N}$ is an irreducible and nondegenerate surface.
\end{prop}

\begin{proof}

Suppose for the sake of contradiction that $\overline{\Im(\varphi)}$ is a nondegenerate curve $C$. Then $\deg C \geq 3$ since $N \geq 3$, and a hyperplane section of $C \subset \P^{N}$ pulls back to a divisor with at least three irreducible components. This contradicts the fact that any divisor $D (\sim_{lin} 2L) \in \d$ has at most two irreducible components since $\NS(A) = \Z[L]$. So the image of $\varphi$ is a surface. \qedhere

\end{proof}

\begin{lem}\label{lem:S}
Let $\varphi: X \dashrightarrow \P^{n}$ be a rational map from a surface $X$ to a projective space of dimension $n \geq 2$, and suppose that its image $S \coloneqq \overline{\Im(\varphi)} \subset \P^{n}$ has dimension 2. Let $\d$ be the linear system corresponding to $\varphi$ (assuming $\d$ has no base components). Then for any $D \in \d$,
\[ \deg \varphi \cdot \deg S \leq D^{2}. \]
\end{lem}

\begin{proof}
The indeterminacy locus of $\varphi$ is a finite set. \qedhere
\end{proof}

\section{Degree bounds}

We now study the numerical properties of the linear series $\d$ constructed above. Keeping the notation as in $\S$\ref{SU}:

\begin{lem}\label{lem:NFC8}
If $\d$ has no fixed components, then
\[ \deg \varphi \cdot \deg S \leq 8. \]
\end{lem}

\begin{proof}
By applying Proposition~\ref{prop:S} and blowing-up $A$ along the collection of two-torsion points $Z$ to resolve some of the base points of $\d$, we arrive at the diagram
\begin{center}
\begin{tikzcd}
\hat{A} \arrow[r, phantom, "\coloneqq"] &[-0.25in] \Bl_{Z}A \arrow[d, swap, "\pi"] \arrow[dr, dashed, "\psi"] & &[-0.25in] \\
& A \arrow[r, swap, dashed, "\varphi"] & S \arrow[r, phantom, "\subset"] & \P^{N}
\end{tikzcd}
\end{center}
The linear system corresponding to $\psi$ has no fixed components, so its divisors are of the form
\[ \hat{D} \sim_{lin} \pi^{\ast}D - \sum_{i=1}^{16} d_{i} E_{i}, \]
where $\hat{D}$ denotes the strict transform of $D$. By Lemma~\ref{lem:S} applied to $\psi$,
\begin{equation}\label{eq:D}
\deg \varphi \cdot \deg S = \deg \psi \cdot \deg S \leq \hat{D}^{2} = 4L^{2} - \sum_{i=1}^{16} d_{i}^{2} \leq 4 \Bigg(2d - \sum_{i=1}^{16} a_{i}^{2} \Bigg) = 8. \qedhere
\end{equation}

\end{proof}

The main work is to treat the case when $\d$ has a fixed divisor $F \not= 0$. In this situation, we may write:
\[ D = F + M \in \d \quad \text{and} \quad \d = F + \b, \]
where $F$ and $M$ are the fixed and movable components of $\d$, respectively. By definition, $\dim \d = \dim \b$. Note that $D \sim_{lin} 2L$ implies $F, M \sim_{alg} L$ for all $M \in \b$. Choose a general divisor $M \in \b$ and write
\[ m_{i} \coloneqq \mult_{p_{i}}M \quad \text{and} \quad f_{i} \coloneqq \mult_{p_{i}}F, \]
so that $d_{i} = m_{i} + f_{i} \geq 2a_{i}$ for all $i$. We claim that $F$ must be symmetric as a divisor. If not, then
\[ \iota(M) + \iota(F) = \iota(D) = D = M + F \quad \text{for all} \quad D \in \d. \]
This implies that $M = \iota(F)$ and $F = \iota(M)$ for all $M \in \b$, which would mean that $M$ must also be fixed, leading to a contradiction. Hence, $F$ must be symmetric, and likewise for all $M \in \b$.

We first need an intermediate estimate:

\begin{prop}\label{prop:E}
Assume $\d$ has a fixed component $F \not= 0$. Keeping the notation as above,
\[ \sum_{i=1}^{16} m_{i}^{2} \geq 2d - 8. \]
\end{prop}

\begin{proof}

The idea here is to use the Kummer construction to push our fixed curve $F$ onto a K3 surface and apply Riemann-Roch. This is analagous to a proof of Bauer's in \cite[Theorem 6.1]{Bauer99}. Consider the smooth Kummer K3 surface $K$ associated to $A$:
\begin{center}
\begin{tikzcd}
E \arrow[r, phantom, "\subset"] &[-0.3in] \hat{A} \arrow[d, swap, "\pi"] \arrow[r, "\gamma"] & \hat{A}/\{ 1, \sigma \} \arrow[r, phantom, "\eqqcolon"] &[-0.3in] K \\
Z \arrow[r, phantom, "\subset"] & A & &
\end{tikzcd}
\end{center}
where $\pi$ is the blow-up of $A$ along the collection of two-torsion points $Z$. Since the points in $Z$ are $\iota$-invariant, $\iota$ lifts to an involution $\sigma$ on $\hat{A}$ and the quotient $K$ is a smooth K3 surface. Let $E_{i}$ denote the exceptional curve over $p_{i} \in Z$, so that $E = \sum_{i=1}^{16}E_{i}$ is the exceptional divisor of $\pi$. Since $F$ is symmetric, its strict transform
\[ \hat{F} = \pi^{\ast}F - \sum_{i=1}^{16} f_{i} E_{i}, \]
descends to an irreducible curve $\bar{F} \subset K$. We claim that
\[ h^{0}(K, \O_{K}(\bar{F})) = 1. \]
In fact, if the linear system $\abs{\O_{K}(\bar{F})}$ were to contain a pencil, then this would give us a pencil of symmetric curves in $\abs{\O_{A}(F)}$ with the same multiplicities at the two-torsion points, which contradicts $F$ being a fixed component of $\d$.

From the exact sequence $0 \rightarrow \O_{K}(-\bar{F}) \rightarrow \O_{K} \rightarrow \O_{\bar{F}} \rightarrow 0$, it follows that $H^{i}(K, \O_{K}(\bar{F})) = 0$ for $i > 0$, so by Riemann-Roch
\[ 1 = h^{0}(K, \O_{K}(\bar{F})) = \chi(\O_{K}, \O_{K}(\bar{F})) = \frac{1}{2} (\bar{F})^{2} + 2 \]
and therefore $(\bar{F})^{2} = -2$. On the other hand, the equality
\[ -4 = 2 (\bar{F})^{2} = (\pi^{\ast}\bar{F})^{2} = (\hat{F})^{2} = F^{2} - \sum_{i=1}^{16}f_{i}^{2} = 2d - \sum_{i=1}^{16} f_{i}^{2} \]
combined with $\sum_{i=1}^{16} f_{i}m_{i} \leq \sum_{i=1}^{16} (\frac{d_{i}}{2})^{2}$ yields
\[ \sum_{i=1}^{16} d_{i}^{2} = \sum_{i=1}^{16} (f_{i}^{2} + m_{i}^{2} + 2 f_{i}m_{i}) \leq 2d+4 + \sum_{i=1}^{16} m_{i}^{2} + \frac{1}{2} \sum_{i=1}^{16} d_{i}^{2}. \]
After rearranging the terms, we find that
\[ \sum_{i=1}^{16} m_{i}^{2} \geq -2d - 4 + \frac{1}{2} \sum_{i=1}^{16} d_{i}^{2} \geq -2d - 4 + 2 \sum_{i=1}^{16} a_{i}^{2} \geq 2d - 8 \]
for a general divisor $D = F + M \in \d$, which is the desired inequality. \qedhere

\end{proof}

As an immediate consequence:

\begin{thm}\label{thm:FC8}
Assume $\d$ has a fixed component $F \not= 0$, and let $\b = \d - F$ be the linear system defining $\varphi: A \dashrightarrow S \subseteq \P^{N}$. Then
\[ \deg \varphi \cdot \deg S \leq 8. \]
\end{thm}

\begin{proof}
As we saw in the proof of Lemma~\ref{lem:NFC8},
\begin{equation}\label{eq:M}
\deg \varphi \cdot \deg S \leq M^{2} - \sum_{i=1}^{16} m_{i}^{2} \leq 2d - (2d - 8) = 8. \qedhere
\end{equation}
\end{proof}

\begin{cor}\label{cor:SQ}
There exists a 4-to-1 rational map $\varphi: A \dashrightarrow \P^{2}$.
\end{cor}

\begin{proof}

Recall that we chose the $a_{i}$ so that $a_{15} = 0 = a_{16}$. From Remark~\ref{rem:K}, it follows that $\varphi: A \dashrightarrow S \subset \P^{N}$ factors through the quotient $A \rightarrow A/\iota$, so $\deg \varphi$ must be even. The surface $S$ is nondegenerate, so $\deg S \geq 2$. By Lemma~\ref{lem:NMT} below, it is impossible for $S$ to be rational together with $\deg \varphi = 2$, so $\{ \deg \varphi = 2, \deg S = 2, 3 \}$ is ruled out by the classification of quadric and cubic surfaces (using the fact that $\rho(A) = 1$).

Together with the upper bound $\deg \varphi \cdot \deg S \leq 8$ given by Lemma~\ref{lem:NFC8} and Theorem~\ref{thm:FC8}, there are two possibilities:
\[ \{ \deg \varphi = 2, \deg S = 4 \} \quad \text{and} \quad \{ \deg \varphi = 4, \deg S = 2 \}. \]
Either of these imply equality throughout~\eqref{eq:D} or~\eqref{eq:M}, so that there is a morphism $\Bl_{Z}A \rightarrow S$ which fits into the diagram:
\begin{center}
\begin{tikzcd}
E_{i} \arrow[r, phantom, "\subset"] &[-0.3in] \Bl_{Z}A \arrow[d, swap, "\pi"] \arrow[r, "\gamma"] \arrow[rd] & K \arrow[d, "\alpha"] \arrow[r, phantom, "\supset"] &[-0.3in] G_{i} \\
& A \arrow[r, dashed, swap, "\varphi"] & S \arrow[r, phantom, "\subset"] & \P^{3}
\end{tikzcd}
\end{center}
where $K$ is the smooth Kummer K3 surface, $\gamma$ is a branched cover of degree 2, and $G_{i} \coloneqq \gamma(E_{i})$.

In the first case where $\deg \varphi = 2$ and $\deg S = 4$, from~\eqref{eq:D} and~\eqref{eq:M} it follows that $d_{15} = 0 = d_{16}$ or $m_{15} = 0 = m_{16}$. This implies that the curves $G_{15}, G_{16}$ are contracted and their images $q_{15}, q_{16}$ under $\alpha$ are double points on $S$ since $\alpha$ is a birational morphism. Projection from a general $(N-3)$-plane containing one but not both of the $q_{i}$ defines a rational map $A \dashrightarrow \P^{2}$ of degree 2 (if $q_{15}$ is a cone point of $S$, pick a general plane passing through $q_{16}$, and vice versa). In the second case where $\deg \varphi = 4$ and $\deg S = 2$, note that $S$ is rational. \qedhere

\end{proof}

This immediately leads to Theorem~\ref{thm:MT}. It is natural to ask what $\irr(A_{d})$ is equal to for a very general polarized abelian surface. At least one can see geometrically:

\begin{lem}\label{lem:NMT}
There are no rational dominant maps $A \dashrightarrow \P^{2}$ of degree 2.
\end{lem}
\begin{proof}
Suppose there exists such a map. We have the following diagram
\begin{center}
\begin{tikzcd}
& & A^{[2]} \arrow[r, "\Sigma"] &[-0.2in] A \\
A \arrow[r, dashed, "f"] & \P^{2} \arrow[r, dashed, swap, "h"] \arrow[ur, dashed, "g"] & K^{[2]}(A) \arrow[u, hook] \arrow[r, phantom, "\eqqcolon"] & \Sigma^{-1}(0)
\end{tikzcd}
\end{center}
where $g$ is the pullback map on 0-cycles and $A^{[2]}$ is the Hilbert scheme of 2 points on $A$. Since the rational map $\Sigma \circ g$ can be extended to a morphism, it must be constant. So $\overline{\Im(g)}$ is contained in a fiber $\Sigma^{-1}(0)$, which is a smooth Kummer K3 surface $K^{[2]}(A)$. Since $g$ is injective, it descends to an injective (and hence birational) map $h: \P^{2} \dashrightarrow K^{[2]}(A)$, yielding a contradiction.
\end{proof}

%%%%%%%%%%%%%%%%
%
%  REFERENCES
%
%%%%%%%%%%%%%%%%

\end{document}